\documentclass[12pt]{amsart}
\usepackage{fullpage, url,epsfig}
\usepackage[backref]{hyperref}
\usepackage{amssymb}
\usepackage[normalem]{ulem}
\usepackage[all]{xy} 
\usepackage{enumerate}

\DeclareFontEncoding{OT2}{}{} 

\usepackage[usenames,dvipsnames]{color}

\newcommand{\C}{{\mathbb C}}
\newcommand{\F}{{\mathbb F}}
\newcommand{\G}{{\mathbb G}}

\newcommand{\PP}{{\mathbb P}}
\newcommand{\Q}{{\mathbb Q}}
\newcommand{\R}{{\mathbb R}}
\newcommand{\Z}{{\mathbb Z}}
\newcommand{\Qbar}{{\overline{\Q}}}

\newcommand{\Xbar}{{\overline{X}}}
\newcommand{\kbar}{{\overline{k}}}

\newcommand{\Pbar}{{\overline{P}}}

\newcommand{\Fbar}{{\overline{\F}}}

\newcommand{\calA}{{\mathcal A}}
\newcommand{\calB}{{\mathcal B}}

\newcommand{\calO}{{\mathcal O}}

\DeclareMathOperator{\inv}{inv}

\DeclareMathOperator{\im}{im}

\DeclareMathOperator{\Gal}{Gal}

\DeclareMathOperator{\Br}{Br}

\DeclareMathOperator{\Div}{Div}
\DeclareMathOperator{\Pic}{Pic}

\DeclareMathOperator{\Proj}{Proj}

\DeclareMathOperator{\res}{res}

\DeclareMathOperator{\NS}{NS}

\newcommand{\et}{{\operatorname{et}}}

\newcommand{\isom}{\simeq}

\newtheorem{theorem}{Theorem}[section]
\newtheorem{lemma}[theorem]{Lemma}

\theoremstyle{definition}

\newtheorem{remark}[theorem]{Remark}

\begin{document}

\title[paper]{Brauer--Manin obstructions on degree 2 K3 surfaces}
\author{Patrick Corn and Masahiro Nakahara}
\address{Department of Mathematics and Computer Science, Emory University, 400 Dowman 
Dr. W401, Atlanta, GA, 30322}
\email{pcorn@mathcs.emory.edu}
\urladdr{http://www.mathcs.emory.edu/~pcorn}

\address{Department of Mathematics, Rice University, 6100 Main St, Houston, TX, 77005}
\email{mn24@math.rice.edu}


\begin{abstract}
We analyze the Brauer--Manin obstruction to rational points on the K3 surfaces over $\Q$ given by 
double covers of $\PP^2$ ramified over a diagonal sextic. After finding an explicit set of 
generators for the geometric Picard group of such a surface, we find two types of infinite 
families of counterexamples to the Hasse principle explained by the algebraic Brauer--Manin obstruction. 
The first type of obstruction comes from a quaternion algebra, and the second type comes from a 
3-torsion element of the Brauer group, which gives an affirmative answer to a question asked by Ieronymou
and Skorobogatov.
\end{abstract}

\maketitle

\section{Introduction}\label{S:introduction}

A variety $X$ over a number field $k$ with points in every completion $k_v$ of $k$ but no $k$-rational points is said to be a counterexample to the Hasse principle. In 1971, Manin \cite{manin1971} identified an obstruction to the existence of rational points using the group $\Br(X):=H^2_{\et}(X,\G_m)_{\text{tors}}$, now known as the {\em Brauer--Manin obstruction}, which he used successfully to explain some of the counterexamples to the Hasse principle that were known at the time. Since then, groundbreaking work of Skorobogatov \cite{skorobogatovbeyond} has shown that the Brauer--Manin obstruction is not the only one for general varieties over number fields (see also \cite{poonen10}, \cite{enriques} ,\cite{ctps}, \cite{smeets}) but Colliot-Th\'el\`ene conjectured \cite{ct2003} that the Brauer--Manin obstruction is the only obstruction to rational points for smooth, projective, geometrically rationally connected varieties; that is, if $X$ is such a variety with points everywhere locally, and there is no Brauer--Manin obstruction to rational points on $X$, then $X$ should have a rational point.

\subsection{Computational evidence} When $X$ is a Del Pezzo surface, in addition to the promising theoretical evidence in favor of the conjecture, there is a growing amount of computational evidence for this conjecture in various cases (e.g. \cite{ctksdiag}, \cite{thesis}, \cite{corndp2}, and \cite{logandp2}). Computing the Brauer--Manin obstruction on a Del Pezzo surface is relatively fast in practice.

For K3 surfaces, the situation is far more unsettled. Skorobogatov has conjectured that the Brauer--Manin obstruction is the only obstruction to both Hasse principle and weak approximation for K3 surfaces over number fields \cite{skoOb}. In light of recent work on diagonal quartics (\cite{bright}, \cite{is2015}) and various Kummer surfaces (\cite{ronaldadam}, \cite{cornK3}, \cite{argentinthesis}, \cite{sz2016}), it seems natural to analyze one of the other main types of K3 surfaces, namely double covers of $\PP^2$ branched over a sextic curve, which are degree 2 K3 surfaces and were also studied in \cite{hvav11}, \cite{hv13}, \cite{mstva}. In this paper, we study the geometry of one of the simplest such families of surfaces:
\begin{equation}\label{k3eq}
X\colon \qquad w^2=Ax^6+By^6+Cz^6
\end{equation}
contained in the weighted projective space $\PP[1,1,1,3]=\Proj k[x,y,z,w]$. The curve defined by $Ax^6+By^6+Cz^6$ on $\PP^2$ is the ramification sextic. We study conditions on the coefficients of $X$ that are sufficient in order to give rise to an obstruction.

\subsection{Supersingular reduction and $\Pic(\Xbar)$} For $X$ as in \eqref{k3eq}, let $\Xbar$ be the base change of $X$ to a fixed algebraic closure $\kbar$ of $k$. There is a filtration $$\Br_0(X)\subseteq\Br_1(X)\subseteq\Br(X),$$ where $\Br_1(X)=\ker(\Br(X)\to\Br(\Xbar))$, and $\Br_0(X)$ is the image of $\Br(k)\to\Br(X)$.  We take advantage of an isomorphism $\Br_1(X)/\Br_0(X)\isom H^1(k,\Pic(\Xbar))$ coming from the Hochschild-Serre spectral sequence, to construct elements of $\Br_1(X)$. Thus our first task is to compute $\Pic(\Xbar)$ for a surface $X$ of the form \eqref{k3eq}. We begin by giving an explicit list of 20 divisors obtained pulling back lines and conics on $\PP^2$ that are tangent to the ramification sextic, which form a sublattice for $\Pic(\Xbar)$. We then show that these divisors generate the entire Picard group using supersingular reduction and elliptic fibrations. Recall that when $\text{char}\ k=0$, the rank of $\Pic(\Xbar)$ is at most 20; we say $X$ has supersingular reduction at a prime $\mathfrak{p}$ of $k$ if the reduction $X_\mathfrak{p}$ has (geometric) Picard rank 22. We use the  intersection properties of the extra divisors coming from the supersingular reduction to show the sublattice we constructed is already saturated in $\Pic(\Xbar)$. To the best of our knowledge, this is the first time supersingular reduction was used to compute the Picard group of the original surface.

Having understood $\Pic(\Xbar)$ as a Galois module, we find coefficients $A,B,C$ for which there is a 2 or 3-torsion element in $\Br_1(X)$ giving an obstruction to rational points.

\begin{theorem}[Theorem \ref{quatthm}]\label{1.1} Let $X/\Q$ be the K3 surface given by
\[
w^2 = 4ax^6 + 2by^6 + 2bc^3 z^6
\]
where $a,b,c$ are nonzero integers. Then the quaternion algebra
\[
{\mathcal A} = \left( \Q(\sqrt{-ac})/\Q, \frac{y^2 + c z^2}{x^2} \right)
\]
represents a class in $\Br(X)$, and if $a,b,c$ satisfies eight conditions (given in the full statement of the theorem below) then there is an algebraic Brauer--Manin obstruction to the Hasse principle on $X$ given by $\calA$.
\end{theorem}

We remark that there are infinitely many triples of integers satisfying the eight conditions--indeed, infinitely many of the form $(-1,1,c)$.

So far all known examples of Brauer--Manin obstruction to the Hasse principle for K3 surfaces have been given by a 2-torsion element in the Brauer group. In \cite{is2015}, Ieronymou and Skorobogatov ask whether it was possible for the odd torsion part of the Brauer group to obstruct the Hasse principle for a K3 surface over a number field. We show that this is possible:

\begin{theorem}[Theorem \ref{oddtor}] Let $X/\Q$ be the K3 surface given by
$$w^2=-3x^6+97y^6+97\cdot 28\cdot 8z^6.$$
Then there is an 3-torsion class $\calA\in \Br(X)$ such that over $k=\Q(\sqrt{-3})$, one can choose a generator $\sigma\in\Gal(k(\sqrt[3]{28})/k)$ such that $\calA_{k}$ can be written as the cyclic algebra
$$\calA_{k} = \left(k(\sqrt[3]{28})/k, \frac{w-\sqrt{-3}x^3}{w+\sqrt{-3}x^3}\right).$$
Then there is an algebraic Brauer--Manin obstruction to the Hasse principle on $X$ given by $\calA$.
\end{theorem}

As with Theorem \ref{1.1}, there are infinitely many such examples given by coefficients satisfying certain congruence and primality conditions (Remark \ref{infrem}). For these surfaces, we show that $\Br_1(X)/\Br_0(X)\isom\Z/3\Z,$ so there are no obstructions coming from the 2-primary part of the algebraic Brauer group.

\begin{remark} The algebra $\calA$ appearing in Theorem \ref{oddtor} can be seen as coming from a double cover of a rational surface. Let $Y$ be the del Pezzo surface of degree 1 defined by the equation
$$u^2=-3r^6+97s^6+97\cdot28\cdot8t^3.$$
in the weighted projective space $\PP[1,1,2,3]=\Proj k[r,s,t,u]$. Consider the double cover $X\to Y$ defined by sending $[w:x:y:z]\mapsto[w:x:y:z^2]$. Then the class $\calA$ comes from the pullback of the class
$$\left(\Q(\sqrt[3]{28})/\Q, \frac{u-\sqrt{-3}r^3}{u+\sqrt{-3}r^3}\right).$$
An easy way to see this is to use \cite[Theorem 1.4]{isz11}, which says that if $\calA$ is odd torsion and fixed by the covering automorphism $z\mapsto -z$, then it comes from a class in $\Br(Y)$.
\end{remark}

\section*{Acknowledgements}

We thank Anthony Varilly-Alvarado for his continuous support and conversations. We thank Jennifer Berg for valuable discussions. We thank Jean-Louis Colliot-Th\'el\`ene, Brendan Hassett, and Olivier Wittenberg for their feedback on preliminary versions of the paper. Much of the computations were done using the computer algebra system MAGMA \cite{magma}

\section{Geometry}

We calculate the geometric Picard group of the surface $X$ over $\Q$ given by the equation 
\begin{equation}\label{k3kbar}w^2 = x^6 + y^6 + z^6.\end{equation}
Let $V$ Let $\Xbar$ be the base change to $\Qbar$ and $\Pbar = \Pic \Xbar$.

\subsection{Generators} Let $s$ be a cube root of $2$ and $\zeta$ be a  primitive 12th root of unity. Set
\begin{align*}
a_5 &= \frac{2\zeta(-\zeta^2+2)(\zeta^4)}{3}&c_5 &= \frac{\zeta(\zeta^2-2)}{3}\\
r_5 &= -\sqrt{3}\zeta(\zeta^2-2)&v_5 &= -2\sqrt{3}\zeta^2
\end{align*}
Consider the following twenty divisors on $\Xbar$:

\begin{align*}
&D_i: w=x^3,\  y=\zeta^{2i+1} z &1\leq i\leq5\\
&D_j: w=y^3,\ z=\zeta^{2j+3}x &6\leq j\leq9\\
&D_{k}: w=z^3,\ x=\zeta^{2k+7}y &10\leq k\leq 12\\
&D_{13}: w=-x^3,\ y=\zeta^3z &\\ 
&D_{14}: w=s(1-2\zeta^2)x^2y+(1-2\zeta^2)y^3,\ x^2+s^2y^2+z^2=0\\
&D_{15}: w=s(1-2\zeta^2)x^2y+(1-2\zeta^2)y^3,\ \zeta^4x^2+\zeta^4s^2y^2+z^2=0\\
&D_{16}: w=-3x^3+s(2-4\zeta^2)x^2y+3s^2xy^2+(2\zeta^2-1)y^3,\\
&\ \ \ \ \ \ -2x^2+s(1-2\zeta^2)xy+s^2y^2+z^2=0\\
&D_{17}: w=(1-2\zeta^2)x^3-3s^2x^2y+s(4\zeta^2-2)xy^2+3y^3,\\
&\ \ \ \ \ \ s^2\zeta^4x^2+s(\zeta^2+1)xy-2\zeta^4y^2+z^2=0\\
&D_{18}: w=-3y^3+s(2-4\zeta^2)y^2z+3s^2yz^2+(2\zeta^2-1)z^3,\\
&\ \ \ \ \ \ x^2-2y^2+s(1-2\zeta^2)yz+s^2z^2=0\\
&D_{19}: w=-3y^3+s(2-4\zeta^2)y^2z+3s^2yz^2+(2\zeta^2-1)z^3,\\
&\ \ \ \ \ \ x^2+2\zeta^2y^2+s(\zeta^2-2)yz-s^2\zeta^2z^2=0\\
&D_{20}: w=r_5x^3+v_5xyz,\ a_5x^2+c_5(y^2+z^2)+yz=0 \\
\end{align*}

The first thirteen divisors $D_1,\ldots,D_{13}$ each arise as one component of a pullback from a line tangent to the ramification sextic in $\PP^2$ while the last seven are pullbacks of conics tangent to the ramification sextic. For example, the line $y=\zeta^3 z$ on $\PP^2_{[x:y:z]}$ is tritangent to $x^6+y^6+z^6=0$, and the pullback is giving by the equation
$$w^2=x^6,\ y=\zeta^3z$$
which splits into two components corresponding to $w=x^3$ or $w=-x^3$. The first component is $D_1$ while the latter is $D_{13}$.

The rest of the divisors were found with the help of Festi's thesis \cite{Festi}, which includes studies the Picard group of the family of K3 surfaces given by $w^2=x^6+y^6+z^6+tx^2y^2z^2$. It is shown in particular, that when $t=0$, the Picard rank is 20 \cite[Remark 3.3.4]{Festi}, but an explicit list of generators is not given. The divisors $D_{14},\ldots, D_{19}$ were obtained by considering $X$ as a double cover of a del Pezzo surface, as explained in \cite[\S3.3.3]{Festi} . The last divisor comes from \cite[p.96]{Festi}. 

Let $d_i$ be the image of $D_i$ in $\Pbar$ and $M\subset \Pbar$ be the subgroup generated all the $d_i$. Computing the intersection pairing on $M$ shows that the discriminant of this lattice is $-432=-2^4\cdot 3^3$. Since the rank of $\Pbar$ is at most $20$ \cite[VIII.3]{bhpv}, $M$ must generate a finite-index subgroup $\Pbar$.  An easy lattice argument shows that the index $[\Pbar \colon M]$ must divide $36$. In fact, this index is $1$:

\begin{theorem}\label{saturated} The geometric Picard group of the surface $X/\Q$ given by
$$w^2=x^6+y^6+z^6$$
is freely generated by the divisor classes $d_1,\ldots,d_{20}$. In particular $\Pic(\Xbar)\isom \Z^{20}$.
\end{theorem}

\begin{proof}
We must show that the sublattice $M\subset \Pbar$ is saturated. Recall that $M$ has discriminant $-2^4\cdot3^3$, so it suffices to show that the maps
\begin{align*}
\phi_2\colon M/2M\to \Pbar/2\Pbar\\
\phi_3\colon M/3M\to \Pbar/3\Pbar
\end{align*}
are injective. Let $G_2$ and $G_3$ be 2-Sylow and 3-Sylow subgroups respectively of $G:=\Gal(K/\Q)$ where $K$ is the smallest Galois extension where all the divisors $d_i$ are defined. Now $G_2$ acts on both $M/2M$ and $\Pbar/2\Pbar$ so we have an induced map $\phi_2^{G_2}\colon (M/2M)^{G_2}\to(\Pbar/2\Pbar)^{G_2}$. Then $(\ker\phi_2)^{G_2}=\ker\phi_2^{G_2}$ will be nonzero whenever $\ker \phi_2$ is nonzero, by \cite[p.64, Proposition 26]{ser77}. Let $d\in M$ whose image $\overline{d}\in M/2M$ is fixed by $G_2$. For $\overline{d}$ to lie in $\ker \phi_2^{G_2}$, we must have $d\cdot d \equiv 0\pmod 4$ and $d\cdot d'\equiv 0\pmod 2$ for all $d'\in M$. A MAGMA computation shows that the only possibility is $d\equiv d_6+d_9+d_{15}\mod 2M$. It suffices to show that $d_6+d_9+d_{15}\notin 2\Pbar$.

We use the method of supersingular reduction. For a prime $p$ such that the reduction $X_{\Fbar_p}$ is smooth, there is an embedding $\Pic\Xbar\to\Pic X_{\Fbar_p}$ (which naturally extends the intersection form on $\Pic \Xbar$, see proof of \cite[Proposition 3.6]{vak3} for example). If $X_{\Fbar_p}$ is supersingular, i.e., when $\Pic X_{\Fbar_p}\isom \Z^{22}$ if $p\neq2$, then intersecting the image of $d_6+d_9+d_{15}$ with the extra divisors in $\Pic X_{\Fbar_p}$ could yield more information. MAGMA computations show that
$$D':= w+2x^3-(2\alpha+1)x^2y+2xy^2+(2\alpha+1)y^3,\ (\alpha+2)x-z+\alpha y,$$
where $\alpha$ is a root of $t^2+t+1$, is a divisor on $X_{\Fbar_5}$ whose class is independent from the $d_i$'s, so 5 is a supersingular prime. Denoting $d'$ the class of $D'$ in $\Pic X_{\Fbar_5}$, we have $(d_6+d_9+d_{15})\cdot d'=1$ which shows that $d_6+d_9+d_{15}\notin 2\Pbar$.

Similarly for $\phi_3$, we are reduced to showing $d_1+d_3+d_5\notin 3\Pbar$. Rewrite \eqref{k3kbar} as
$$(w+x^3)(w-x^3)=(y^3+iz^3)(y^3-iz^3).$$
There is a morphism $X\to\PP^3_{[r:s:u:v]}$ given by $[w:x:y:z]\mapsto [w+x^3:w-x^3:y^3+iz^3:y^3-iz^3]$, whose image is contained in the quadric $Q$ defined by $rs=uv$. The projection to one of the rulings on $Q$ will define an elliptic fibration on $X$. Indeed, up to a choice of ruling, the fiber above $[\alpha,\beta]\in\PP^1$ will be
\begin{align*}
\alpha(w+x^3)=\beta(y^3+iz^3)\\
\beta(w-x^3)=\alpha(y^3-iz^3)
\end{align*}
When $\alpha\beta\neq0$ and $\alpha\neq\pm \beta$, this is isomorphic to
$$2\alpha\beta x^3=(\beta^2-\alpha^2)y^3+(\beta^2+\alpha^2)iz^3,$$
which is a smooth integral curve of genus 1. When $\alpha=0$, the fiber is precisely $d_1+d_3+d_5$. Hence $d_1+d_3+d_5$ is a fiber of an elliptic fibration, so it cannot be divisible (see \cite[Proposition 11.1.5(iii)]{huyk3}).
\end{proof}

\subsection{An alternate approach}
We show how one can also compute $\Pic \Xbar$ without the use of the divisor $D_{20}$ from Festi's thesis. This section is unnecessary for the main results of the paper, but we include it in case it is of use to the reader.

Consider the lattice $M'$ generated by the divisors $d_1,\ldots,d_{19}$ along with the more obvious divisor
$$d'_{20}:w=z^3,\ x=\zeta^9y.$$
These divisors still generate a rank $20$ sublattice in $M'\subset\Pbar=\Pic \Xbar$. In fact, one can show that $M' \ne \Pbar$, and also abstractly compute the full Picard group without using the divisor $d_{20}$.

By \cite[Prop. 1.6.1]{nikulin}, $\Pbar^\vee/\Pbar\isom T^\vee/T$ where $T:= \NS(X)^\perp\subset H^2(X(\C),\Z)$ is the transcendental lattice. Note in this case the bilinear form on $T$ is given by a positive definite $2\times2$ matrix. Hence we can write $\Pbar^\vee/\Pbar\isom \Z/n\Z\times\Z/m\Z$ for some $m,n\in\Z$. However the discriminant of $M'$ is $-2^4\cdot 3^5$ and the discriminant group $M'^\vee/M'$ is isomorphic to
$$(\Z/4\Z)^2\times (\Z/3\Z)^3\times \Z/9\Z.$$
This implies that $M'\neq\Pbar$. The method of supersingular reduction used in the proof of Theorem \ref{saturated} can determine the ``missing" divisor. For $r$ and $p$ primes, define
\[
S_{r,p} = \{ d \in M' \colon d \cdot x \equiv 0 \ \text{mod $r$, for all $x \in \Pic X_{\overline {\mathbb 
F}_p}$} \}.
\]
Then for all $p$, any $d \in M'$ which is divisible by $r$ in $\Pbar$ will lie in $S_{r,p}$.

For our surface $X$, recall that $5$ is a supersingular prime. Writing down obvious divisors on $X_{\overline {\mathbb F}_5}$ and computing intersections, we get
\begin{align*}
S_{2,5} &= 2M' \\
S_{3,5} &\subseteq 3M' + \Z (d_1 + d_4 + d_6 + d_9 + d_{10} + d_{20}' - (d_2 + d_3 + d_7 + d_8 + d_{11} + d_{12}))
\end{align*}

Let $d = d_1 + d_4 + d_6 + d_9 + d_{10} + d_{20}' - (d_2 + d_3 + d_7 + d_8 + d_{11} + d_{12})$. Since $M'\neq\Pbar$, the above computations show that $M'+\Z[d/3]\subseteq \Pbar$. One can then show saturation as in proof of Theorem \ref{saturated}. This abstract computation of the Picard group is also sufficient for the purpose of doing arithmetic on our surface $X$, and in particular for computing the group $H^1(k,\Pbar)$.

\section{Arithmetic}

\subsection{Fields of definition} We are interested in degree 2 K3 surfaces $X$ over $\Q$ of the form
\begin{equation}\label{abc}
w^2 = Ax^6 + By^6 + Cz^6 \, \, (A,B,C \in \Q)
\end{equation}
Let $\alpha, \beta, \gamma$ be sixth roots of $A,B,C$ respectively. The results of the previous section show that replacing $x$ by $\alpha x$, $y$ by $\beta y$, $z$ by 
$\gamma z$ in the divisors $d_1, \ldots, d_{20}$ gives a set of divisor classes generating $\Pbar = \Pic \Xbar$. The minimal field of definition of these twenty divisors is 
the extension $\Q(s,\zeta,\alpha/\beta, \beta/\gamma, \alpha^3)$. Generically this has degree $3 \cdot 4 \cdot 6 \cdot 6 \cdot 2 = 864$ over $\Q$. Let $G$ be the generic 
Galois group of this extension. 

\subsection{Algebraic Brauer--Manin obstructions} For the definition of the Brauer--Manin obstruction, see \cite{manin1971} and also \cite{skorobogatovtors}. Our goal is to give some examples of the algebraic Brauer--Manin obstruction on the surfaces (\ref{abc}). More ambitiously, one might hope to eventually carry out a full study of the algebraic Brauer--Manin obstruction on the family (\ref{abc}), as in \cite{ctksdiag} and \cite{corndp2}.

Suppose that $X$ has points everywhere locally; then $\Br(k)$ embeds naturally in $\Br_1(X)$ (elements of the image are called {\em constant algebras}) and $\Br_0(X)$ denotes its image. The idea is to make the key isomorphism $\Br_1(X)/\Br_0(X) \to H^1(k,\Pbar)$ explicit. As in \cite{cornK3}, our program proceeds by first analyzing $H^1(H,\Pbar)$ for all the subgroups $H$ of $G$, then identifying those $H$ which will give rise to nonconstant cyclic algebras in $\Br_1(X)$. Since $|G|=864=2^5\cdot3^3$, only 6-primary elements appear in $\Br_1(X)/\Br_0(X)$. We first consider quaternion algebras, and then 3-torsion elements in section \ref{odd}.

A MAGMA computation shows that $H^1(G,\Pbar) = 0$, so there is no algebraic Brauer--Manin obstruction on the generic surface (\ref{abc}). This is in marked contrast to the del Pezzo surfaces in \cite{ctksdiag} \cite{corndp2}, where $\Br_1(X)/\Br_0(X)$ is generically trivial.

In order to construct explicit cyclic algebras, we use the following standard lemma.

\begin{lemma}\label{cyclem} Let $k$ be a number field, let $X/k$ be a smooth projective variety, and let $(L/k,f)$ be a cyclic algebra in $\Br k(X)$, where $L/k$ is cyclic and $f \in 
k(X)$. Then $(L/k,f)$ is in $\Br(X)$ if and only if the divisor of $f$ equals the norm from $L/k$ of a divisor $D$ on $X_L$; if $X$ is everywhere locally soluble, it is a constant algebra if and only if we 
can take $D$ to be principal. \end{lemma}

{\em Proof:} This is a standard result; see \cite[Proposition 2.2.3]{thesis} for a proof.

The lemma can be stated more precisely using \cite[Theorem 3.3]{va08}. Let $N_{L/k}:\Div X_L\to\Div X_k$ and $\overline{N}_{L/k}:\Pic X_L \to \Pic X_k$ be the usual norm maps. Let $\Delta:\Pic X_L\to\Pic X_L$ be given as acting by $1-\sigma$ where $\sigma$ is a generator for $\Gal(L/k)$. Then we have an isomorphism
$$\ker \overline{N}_{L/k}/\im\Delta\isom \Br_{\text{cyc}}(X,L).$$
where $\Br_{\text{cyc}}(X,L)$ denotes cyclic algebras in $\Br_1(X)$ which are split by $L$. The map is defined by sending $[D]$ to $[(L/k,f)]$ where $f\in k(X)$ is any function such that $\text{div}(f)=N_{L/k}(D)$.

\subsection{The explicit quaternion algebra}\label{quatalg}
If we take the short exact sequence of Galois modules
$$0\to  \Pbar\xrightarrow{\times 2}\Pbar\to \Pbar/2\Pbar\to 0$$
and consider its long exact sequence in Galois cohomology, we get 
\[
H^1(H,\Pbar)[2] \cong \frac{(\Pbar/2\Pbar)^H}{\Pbar^H/2\Pbar^H}.
\]
Hence we search for subgroups $H$ which fix certain elements in $\Pbar/2 \Pbar$ without fixing any of their representatives in $\Pbar$. In practice, we search first for 
groups $H$ with a nonzero element in $H^1(H,\Pbar)[2]$ that becomes zero upon restriction to an index-$2$ subgroup of $H$. We find four subgroups $H$ of order $288$ 
with this property.

Let $H_1$ be the mod-2 stabilizer in $G$ of the divisor class 
\begin{equation}\label{divisor}
d_1 - d_4 = \begin{pmatrix} w = \alpha^3 x^3 \\ \beta y = \gamma \zeta^3 z \end{pmatrix} - \begin{pmatrix} w = \alpha^3 x^3 \\ \beta y = 
\gamma \zeta^9 z \end{pmatrix}.
\end{equation}
Then the orbit of $d_1 - d_4$ under $H_1$ is $\{ \pm (d_1 - d_4) \}$. 

Let $H_2$ be the stabilizer of the class $d_1-d_4$ inside $H_1$.

\begin{lemma} Let $c=C/B$. Then the Galois group of the field of definition of $d_1, \ldots, d_{20}$ is a subgroup of 
$H_1$ if and only if $c$ is a cube in $\Q$. In this case,
$$\left( \Q(\sqrt{-Ac})/\Q,\ \dfrac{y^2 + c z^2}{x^2} \right)$$
represents a quaternion algebra in $\Br(X)$.
\end{lemma}

\begin{proof} We construct a quaternion algebra in $\Br(X)$ using Lemma \ref{cyclem} as follows.

The first step in applying the lemma is to find a divisor class defined over a suitable quadratic extension whose norm is zero in $\Pbar$; in this case, we can take 
$L/k$ to be the field of definition of $d_1 - d_4$, the fixed field of $H_2$, and the class $d_1-d_4$ has norm zero. The next step (usually the most computationally 
difficult) is to find a divisor defined over $L$ whose divisor class is $d_1 - d_4$. Note that the divisor $D_1 - D_4$, the difference between the two lines on the right 
side of (\ref{divisor}), is not defined over $L$. In fact, its orbit has four elements, namely 
\[
D_1 - D_4, D_4 - D_1, D_1' - D_4', D_4' - D_1',
\]
where $D_i'$ is $D_i$ with $x^3$ replaced by $-x^3$. So $D_1 - D_4$ is defined over a quartic extension of the ground field. 

It is not hard to see that $H_1$ is the stabilizer inside $G$ of the element $(\beta/\gamma)^2$, a cube root of $B/C$. Hence the Galois group of the field of definition of the divisor classes is a subgroup of $H_1$ if and only if $C/B$ is a cube in $\Q$. So let $C = Bc^3$. The quartic extension over which $D_1-D_4$ is defined is $\Q(\alpha^3,\zeta^3 \gamma/\beta) = \Q(\sqrt{A}, \sqrt{-c})$. Call this extension $M$.

The fixed field of $H_2$ is the extension fixed by the automorphism $\alpha^3 \mapsto -\alpha^3$, $\zeta^3 \gamma/\beta \mapsto \zeta^9 \gamma/\beta$ (which sends $D_1$ to $D_4'$ and $D_4$ to $D_1'$--this is because $D_1 - D_4 = D_4' - D_1'$). So $L$, the fixed field of $H_2$, is the quadratic extension $\Q(\sqrt{-Ac})$.

In order to find a divisor in the class of $d_1 - d_4$ defined over $L$, we start with $D_1-D_4$, defined over $M$, and search for a function $g \in k(X_M)$ such that
\begin{equation}\label{cocycle}
D_1 - D_4 + (g) = D_4' - D_1' + (\sigma g),
\end{equation}
where $\sigma$ is the nontrivial element of $\Gal(M/L)$.

Simplifying equation (\ref{cocycle}) gives
\[
\left( \frac{\sigma g}{g} \right) = D_4' + D_4 - (D_1' + D_1) = \left( \frac{y - z \zeta^9 \gamma/\beta}{y - z \zeta^3 \gamma/\beta} \right),
\]
which suggests letting $g = \dfrac{y-z\zeta^3 \gamma/\beta}{x} = \dfrac{y-z\sqrt{-c}}{x}$.

Now $f$ should be a rational function whose divisor is 
\[
D_1 - D_4 + (g) + \tau(D_1 - D_4 + (g)),
\]
where $\tau$ is the nontrivial element of $\Gal(L/\Q)$. Extend $\tau$ to the automorphism fixing $\sqrt{A}$ and sending $\sqrt{-c}$ to $-\sqrt{-c}$; then 
$\tau(D_1) = D_4$, $\tau(D_4) = D_1$, and $\tau(g) = \dfrac{y+z\sqrt{-c}}{x}$. Putting this all together gives the natural choice
\[
f = \frac{y^2 + c z^2}{x^2}.
\]
So our quaternion algebra in $\Br(X)$ is $\left( \Q(\sqrt{-Ac})/\Q,\ \dfrac{y^2 + c z^2}{x^2} \right)$.\end{proof}

\section{Explicit counterexamples to the Hasse principle}

By the results of the previous section, the surface $X$ given by the equation $w^2 = Ax^6 + By^6 + Bc^3 z^6$ has a quaternion algebra in $\Br(X)$ given by 
\[
{\mathcal A} = \left( \Q(\sqrt{-Ac})/\Q,\ \frac{y^2 + c z^2}{x^2} \right).
\]
To find examples where $\mathcal A$ exhibits a Brauer--Manin obstruction to rational points, we search for integers $A,B,c$ such that
\[
\sum_v {\rm inv}_v \calA(P_v) = 1/2
\]
for all $(P_v) \in X({\mathbb A}_{\Q})$. Note that ${\rm inv}_v {\mathcal A}(P_v) = [-Ac, f(P_v)]_v$, where $f = (y^2 + c z^2)/x^2$ and $[,]_v$ is the Hilbert 
symbol, written additively: $[a,b]_v = 0$ if the quadratic form $x^2 -ay^2 -bz^2$ represents $0$ over $\Q_v$, and $1/2$ otherwise. (This formula holds if $P_v$ is not in the zero locus of the 
numerator or denominator of $f$; otherwise we must use another rational function $f'$ obtained from $f$ by multiplying it by the norm of some rational function in $k(X_L)$.)

\begin{theorem}\label{quatthm} Let $a$, $b$, $c$ be odd integers satisfying the following conditions:
\begin{enumerate}[(i)]
\item For every prime $p>3$ dividing $a$ or $b$, $\nu_p(a)$ and $\nu_p(b)$ lie in $\{ 1,2,4,5 \}$
\item $c$ is squarefree
\item $a > 0$ or $b > 0$
\item $a$ or $-ac \equiv 1$ mod $3$
\item If $p$ is a prime divisor of $a$, then $p|c$ and $\left( \dfrac{2b}{p} \right) = 1$
\item If $p$ is a prime divisor of $b$, then $\left( \dfrac{a}{p} \right) = 1$ and $\left( \dfrac{-c}{p} \right) = 1$
\item If $7 | c$, then we do not have $4a \equiv 2b \equiv 3,5$, or $6$ mod $7$
\item The triple ($a$ mod 8, $b$ mod 8, $c$ mod 8) equals $(3,1,1), (3,3,3), (3,3,7), (3,5,1), (3,7,3)$, $(3,7,7), (5,1,1), 
(5,1,5), (5,1,7), (5,3,1), (5,3,5), (5,3,7), (5,5,1), (5,5,5), (5,5,7), \ \ \ \ $ $(5,7,1), (5,7,5), (5,7,7), (7,1,3), (7,1,7), (7,3,5)$, $(7,5,3), (7,5,7)$, or $(7,7,5)$ 
\end{enumerate}
Then the surface $w^2 = 4ax^6 + 2by^6 + 2bc^3 z^6$ over $\Q$ is a counterexample to the Hasse principle explained by the algebraic Brauer--Manin obstruction.
\end{theorem}

\smallskip

{\em Proof:} Let $X$ be the surface given in the theorem. We will show
\begin{enumerate}[(a)]
\item $X$ has points everywhere locally
\item $\inv_v {\mathcal A}(x_v) = 0$ for all $v \ne 2$ and $x_v \in X(\Q_v)$
\item $\inv_2 {\mathcal A}(x_2) = 1/2$ for all $x_2 \in X(\Q_2)$
\end{enumerate}

\subsection{Proof of (a)} At $v=\infty$, condition (iii) guarantees a real point. Next suppose $v$ corresponds to a prime $p \nmid 6ab$. If $p > 13$ there is always a smooth point on the genus-2 curve 
$w^2 = 4ax^6 + 2by^6$ mod $p$ by the Hasse-Weil bound, which we can lift to a $\Q_p$-point by Hensel's lemma. For $p=13$ there is a point on this curve if and only if at least one of $4a$, 
$2b$, $4a+2b$, and $4a-2b$ is a square mod $13$, which always happens. For $p=7$ there is a point on this curve if and only if at least one of $4a$, $2b$, or $4a+2b$ is a square mod $7$, which 
happens as long as $4a$ and $2b$ are not congruent to each other and to a nonsquare mod $7$, as in condition (vii). If $7 \nmid c$, then it is not hard to check that the surface will always 
have a point mod $7$. For $p=5$ and $p=11$, being a sixth power is equivalent to being a square, and the curve $w^2 = 4ax^2 + 2by^2$ has $p+1$ points mod $p$, so $w^2 = 4ax^6 + 2by^6$ has 
$\Q_p$-points by Hensel's lemma. 

We finish the proof of (a) by checking primes dividing $6ab$. If $p|a$, then condition (v) guarantees a nontrivial point $(\sqrt{2b} \colon 0 \colon 1 \colon 0)$ on $X(\Q_p)$. If $p|b$, then 
the first part of condition (vi) guarantees a nontrivial point $(2\sqrt{a} \colon 1 \colon 0 \colon 0)$ on $X(\Q_p)$. For $p=3, p \nmid ab$ we either have the point $(2\sqrt{a} \colon 
1 \colon 0 \colon 0) \in X(\Q_3)$ if $a \equiv 1$ mod $3$, the point $(\sqrt{2b} \colon 0 \colon 1 \colon 0)$ if $2b \equiv 1$ mod $3$, or the point $(\sqrt{4a+2b} \colon 1 \colon 1 \colon 0)$ 
otherwise.

For $p=2$, all the triples in condition (viii) satisfy one of the following three conditions:
\begin{itemize}
\item $-c^3 \equiv 1$ mod $8$
\item $\dfrac2{b} - c^3 \equiv 1$ mod $8$
\item $\dfrac{2-2a}{b} - c^3 \equiv 1$ mod $8$
\end{itemize}

Note that odd rational integers in $\Q_2$ are sixth powers if and only if they are $1$ mod $8$. So: if the first condition is satisfied, then there is a point $(0 \colon 0 \colon 
\sqrt[6]{-c^3} \colon 1) \in X(\Q_2)$. If the second condition is satisfied, then there is a point $\left( 2 \colon 0 \colon \sqrt[6]{\frac2{b} - c^3} \colon 1 \right) \in X(\Q_2)$. 
If the third condition is satisfied, then there is a point $\left( 2 \colon 1 \colon \sqrt[6]{\frac{2-2a}{b} - c^3} \colon 1 \right) \in X(\Q_2)$.

\subsection{Invariant computations and shortcuts} Next we prove (b). When we compute invariants, we will avail ourselves of several shortcuts. First, we can certainly substitute any rational 
function of the form $\dfrac{y^2 + c z^2}{\ell(x,y,z)^2}$ for $f$, since this function differs from $f$ by a square in $k(X)$. So as long as $y(P_v)^2 + c z(P_v)^2$ is nonzero, we 
have 
\begin{equation}\label{form1}
\inv_v {\mathcal A}(P_v) = [-ac, y(P_v)^2 + c z(P_v)^2]_v.
\end{equation}
Since $y^2 + c z^2$ is a norm from $\Q_v(\sqrt{-c})$ to $\Q_v$, we can further simplify:
\begin{equation}\label{form2}
\inv_v {\mathcal A}(P_v) = [a, y(P_v)^2 + c z(P_v)^2]_v.
\end{equation}
There is one more formula we will use in our computations. Note that the equation of the surface can be rewritten as $2b(y^2 + c z^2)(y^4 - c y^2 z^2 + c^2 z^4) = w^2 - 4ax^6$, 
so 
\begin{align*}
[a,2b]_v + [a,y(P_v)^2 + c z(P_v)^2]_v + [a,y(P_v)^4 - c y(P_v)^2 z(P_v)^2 + c^2 z(P_v)^4]_v & \\
= [a,w(P_v)^2 - 4a x(P_v)^6]_v &= 0
\end{align*}
because $w(P_v)^2 - 4ax(P_v)^6$ is clearly a norm from $\Q_v(\sqrt{a})$ to $\Q_v$. Since these symbols take values in $\frac12 \Z/\Z$, we get
\begin{equation}\label{form3}
\inv_v {\mathcal A}(P_v) = [a,2b]_v + [a,y(P_v)^4 - c y(P_v)^2 z(P_v)^2 + c^2 z(P_v)^4]_v.
\end{equation}
We will use this when $y(P_v)^2 + c z(P_v)^2 = 0$. Note that if $y(P_v)^2 + c z(P_v)^2$ and $y(P_v)^4 - c y(P_v)^2 z(P_v)^2 + c^2 z(P_v)^4$ are both $0$, then substituting 
in $z(P_v)^2 = -c y(P_v)^2$ gives $3y(P_v)^4 = 0$, so $y(P_v) = 0$, and similarly $z(P_v) = 0$. (Note that the same holds mod $p$ if $p \nmid 3 c$.) This implies that $w(P_v)^2 = 4a 
x(P_v)^6$, so $a$ would be a square in $\Q_v$. Hence either $-ac$ would be a square in $\Q_v$ or $y^2 + c z^2$ would be the norm of $y + \sqrt{-ac}(z/\sqrt{a})$ from 
$\Q_v(\sqrt{-ac})$ to $\Q_v$, so we would automatically have that $\inv_v {\mathcal A}(P_v) = 0$ for any $P_v$. So one of the three formulas above will compute the invariant unless it is 
automatically $0$. 

\subsection{Proof of (b)} At $v=\infty$, we need either that $-ac > 0$ or $y(P_v)^2 + c z(P_v)^2 > 0$ for all $P_v \in X(\Q_v)$. If $a>0$, then the first inequality holds if $c < 0$ 
and the second holds if $c > 0$. So (given that condition (iii) holds) we must only check that the second inequality also holds when $a < 0$ and $b > 0$. But in this case, we have that 
$w(P_v)^2 - 4ax(P_v)^6 = 2b(y(P_v)^6 + c^3 z(P_v)^6)$, and the left side is positive, so $y(P_v)^6 + c^3 z(P_v)^6 > 0$; hence $y(P_v)^2 + c z(P_v)^2 > 0$ as well.

Next suppose that $v$ corresponds to a prime $p \nmid 6abc$. Then $-ac$ is a unit in $\Q_p$, and if $y(P_v)^2 + c z(P_v)^2$ is a unit in $\Q_p$, we see that $\inv_v {\mathcal 
A}(P_v) = 0$ automatically. If it is not a unit in $\Q_p$, then $-c$ is a square mod $p$, and we get $w(P_v)^2 \equiv 4a x(P_v)^6$ mod $p$. If $w(P_v)$ or $x(P_v)$ is not 0, then $a$ is a square mod $p$. So then $-ac$ 
is a square mod $p$, which implies $\inv_v {\mathcal A}(P_v) = 0$ as well. If $w(P_v)=x(P_v)=0$, then we can assume $y(P_v)=1,z(P_v)=\pm\sqrt{-1/c}$. Thus $\inv_v {\mathcal A}(P_v)=[a,2b]_v+[a,3]_v=0$ since $a,2b,3$ are all units in $\Q_p$.

Now suppose $p \mid a$, $p > 3$. Then if $w(P_v)$ is divisible by $p$, condition (v) shows that $y(P_v)$ is as well. But then conditions (1), (2), and (5) force $x(P_v)$ and $z(P_v)$ to be divisible by 
$p$. So we may assume $w(P_v)$ (and $y(P_v)$) are invertible mod $p$, so $y(P_v)^2 + c z(P_v)^2$ is a unit and a square in $\Q_p$, so the invariant is automatically zero.

Now suppose $p \mid b$, $p > 3$. Then condition (vi) shows that $-ac$ is a square mod $p$, so that the invariant is automatically zero.

Suppose $p \mid c$, $p > 3$. We may assume $p \nmid ab$ as well. There are two cases: if $y(P_v)$ is divisible by $p$, then so is $w(P_v)^2 - 4a x(P_v)^6$, so that $a$ is a square mod $p$ and the 
invariant is automatically zero (by formula (\ref{form2})); but if $y(P_v)$ is a unit in $\Q_p$, then $y(P_v)^2 + c z(P_v)^2$ is a unit and a square in $\Q_p$, whence the invariant is again 
automatically zero.

Finally, if $p=3$, then $a$ or $-ac$ is a square in $\Q_3$, so we are done by formulas (\ref{form1}) and (\ref{form2}).

\subsection{Proof of (c)} We must analyze $X(\Q_2)$ in order to compute the invariant at any point. So take a point $(2w_2 \colon x_2 \colon y_2 \colon z_2) \in X(\Q_2)$, scaled so that the coordinates are 
elements of $\Z_2$, not all even. (We have written the first coordinate as $2w_2$, $w_2 \in \Z_2$, because clearly the first coordinate must be even.) Since $2b(y_2^6 + c^3 z_2^6) = 4w_2^2 - 4ax_2^6$ is 
divisible by $4$ and $b$ is odd, $y_2$ and $z_2$ must have the same parity. Note that if $y_2$ and $z_2$ are both even, then we get $w_2^2 \equiv ax_2^6$ mod $2^6$, which can only happen if $x_2$ is even 
because $a$ is never congruent to $1$ mod $8$ (see condition (viii)). This is a contradiction, so we may assume that $y_2$ and $z_2$ are both odd.

Now we use formula (\ref{form3}) to compute the invariant. We get
\[
[a,2b]_2 + [a, y_2^4 - c y_2^2 z_2^2 + c^2 z_2^4]_2 = [a,2b]_2 + [a,2-c]_2
\]
because both arguments of the second norm residue symbol are odd, hence we need only know them modulo $8$. It is easy to check that the sum of these two symbols is $1/2$ for every triple in condition (viii). 
So we are done. \hfill $\Box$

\medskip

\begin{remark} There are infinitely many triples $(a,b,c)$ satisfying the conditions of the theorem. For instance, let $\gamma$ be any squarefree integer congruent to $7$ or $19$ mod $24$; 
then the triple $(-1,1,c)$ satisfies the conditions of the theorem, so the surface
\[
w^2 = -4x^6 + 2(y^6 + c^3 z^6)
\]
is a counterexample to the Hasse principle explained by the algebraic Brauer--Manin obstruction.
\end{remark}

\section{Odd torsion obstruction to the Hasse principle}\label{odd}

Now we focus on the case of 3-torsion Brauer elements. There are three index-2 subgroups $H$ of $G$ for which the 3-torsion in $H^1(H,\Pic(\Xbar))$ is nontrivial. In each of these cases, $H^1(H,\Pbar)$ is isomorphic to $\Z/3\Z$. When the surface is defined over $\Q$, each of these three subgroups correspond to one of the following restrictions on the coefficients $A,B,C$:
\begin{enumerate}
\item $-3A$ is a square.
\item $-3B$ is a square.
\item $-3C$ is a square.
\end{enumerate}
Condition (2) is equivalent to saying $X/\Q$ can be written in the form
\begin{equation}
\label{k3}
w^2=-3x^6+B'y^6+C'z^6
\end{equation}
for some $B',C'\in\Q$ such that $(\zeta,\sqrt[3]{2},\sqrt[6]{B'},\sqrt[6]{C'})$ is Galois general. We will study the obstruction on the surfaces of the form \eqref{k3}.

\subsection{The explicit cyclic algebra.} Assume $X/\Q$ is given by \eqref{k3} and let $\calA\in\Br(X)$ such that its class in $\Br_1(X)/\Br_0(X)$ is a generator. Instead of finding a cyclic algebra representative for $\calA$ over $\Q$, we find one over the larger field $K=\Q(\sqrt{-3})$. The larger field will not be a problem in computing invariants, since $[K:\Q]$ is coprime to 3. Let $L=K(\gamma^2/\beta^2)$ which is a cyclic extension, and fix a generator $\sigma\in\Gal(L/K)$. We proceed as before and use Lemma \ref{cyclem}. A simple computation shows that
$$\ker{\overline{N}_{L/K}}/\im\Delta\isom\Z/3\Z$$ with generator $d_1+d_4-d_{13}-d_{13}'\in \Pic X_L$, where $d_{13}'$ is the class of the divisor
\begin{equation}
D_{13}'=\begin{pmatrix} w = -\sqrt{-3} x^3 \\ y = \zeta^9 \gamma/\beta z \end{pmatrix}.
\end{equation}
Then the image of $D_1+D_4$ under $N_{L/K}$ in $\Div X_K$ is
\begin{equation}
\begin{pmatrix} w = \sqrt{-3} x^3 \\ y^6 =  z^6 \end{pmatrix}.
\end{equation}
The image of $D_{13}+D_{13}'$ under $N_{L/K}$ is
\begin{equation}
\begin{pmatrix} w = -\sqrt{-3} x^3 \\ y^6 = z^6  \end{pmatrix}.
\end{equation}
One easily sees that
$$f=\frac{w-\sqrt{-3}x^3}{w+\sqrt{-3}x^3}$$
satisfies $\text{div}(f)=N_{L/K}(D_1+D_4-D_{13}-D_{13}')$ so Lemma \ref{cyclem} gives $\calB:=(L/K,f)\in \Br(X_K)$. Since $[K:\Q]=2$, the natural restriction $H^1(\Q,\Pbar)[3]\to H^1(K,\Pbar)[3]$ is injective. A MAGMA computation shows $H^1(K,\Pbar)\isom\Z/3\Z$ also, so the restriction map is in fact an isomorphism and the image of $\calB$ generates $\Br_1(X_K)/\Br_0(X_K)$. Hence we can assume $\calB=\calA_K$.

\subsection{The counterexample}

\begin{theorem}\label{oddtor} The surface $X/\Q$ given by the equation
$$w^2=-3x^6+97y^6+97\cdot 28\cdot 8z^6$$
is a counterexample to the Hasse principle explained by a 3-torsion Brauer element in the algebraic Brauer group.
\end{theorem}

\begin{proof}
One easily checks that $X$ has local points everywhere. To compute evaluation maps, we first extend the ground field to $K$ so that we can use the algebra representation $\calA_K=(L/K,f)$. Let $X_0$ be the open subset defined by $w^2+3x^6\neq0$. Then the implicit function theorem gives that $X_0(K_v)$ is dense in $X(K_v)$ for all places $v$. It is well known that the local evaluation maps or Brauer elements are continuous, so it suffices to compute the evaluation maps on $X_0$, where the rational function $f$ is defined everywhere. Let $p\in \Z$ be a prime. Since our cyclic algebra representation is only valid on $K$, we first study the evaluation map $\inv_v \calA_K(-):X_0(K_v)\to\Q/\Z$ restricted to the image of $X_0(\Q_p)\hookrightarrow X_0(K_v)$ for some prime $v$ lying above $p$. In the case $K_v=\Q_p$, we immediately get $\inv_p \calA (-)=\inv_v \calA_K(-)$. In the case $[K_v:\Q_p]=2$, we can still discern information on $\inv_v \calA_K(-)$ by using the commutative diagram \cite[Chapter XIII, \S3, Proposition 7]{serrelocalfields}
\begin{align}\label{diag}
\xymatrixcolsep{4pc}\
\xymatrix{
\Br(\Q_p) \ar[d]^{\inv_p} \ar[r]^{\res_{K_v/\Q_p}} & \Br(K_v) \ar[d]^{\inv_v} \\\
\mathbb{Q}/\mathbb{Z} \ar[r]^{\cdot 2} & \mathbb{Q}/\mathbb{Z}}
\end{align}
We now compute the invariant maps for all finite primes. For the infinite place, the evaluation map is trivial since $\Br(\R)\isom\Z/2\Z$. Upon fixing a third root of unity, for any point $P_v\in X_0(K_v)$ we have $\inv_v \calA_K(P_v)=[f(P),28\cdot 8]_v=[f(P),28]_v$, where $[\ ,\,]_v$ is the norm residue symbol.

\subsubsection{$p\neq 2,7$}

Let $v=v_\mathfrak{p}$ be a place corresponding to a prime $\mathfrak{p}$ lying above $p$. If $28$ is a cube in $K_v$, then $N_{LK_v/K_v}$ is the identity map, so the norm residue symbol is trivial for any point in $X(K_v)$. Suppose $28$ is not a cube in $K_v$; in particular this means $p\neq3,97$. Let $P=[w:x:y:z]\in X(\Q_p)$. We can homogenize so that $w,x,y,z\in\Z_p$ and one of them is in $\Z_p^\times$.

Case 1: $\mathfrak{p}$ divides neither $w-\sqrt{-3}x^3$ nor $w+\sqrt{-3}x^3$. Then $f(P)$ is a unit in $\calO_v$, so it must be a cube in $K_v(\sqrt[3]{28})$. Hence $[f(P),28]_v=0$.

Case 2: $\mathfrak{p}$ divides either $w-\sqrt{-3}x^3$ or $w+\sqrt{-3}x^3$. This means $\mathfrak{p}$ divides $97y^6+97\cdot28\cdot8z^6$ which can only happen if $\mathfrak{p}\mid y,z$. Consequently, $v_\mathfrak{p}(97y^6+97\cdot28\cdot8z^6)\equiv 0 \mod 6$. Moreover, $\mathfrak{p}$ must divide exactly one of $w-\sqrt{-3}x^3$ or $w+\sqrt{-3}x^3$, since otherwise $\mathfrak{p}\mid w,x$. Thus $v_\mathfrak{p}(f(P))\equiv 0 \mod 6$, and $[f(P),28]_v=[u,28]_v$ for some unit $u\in \calO_v$ so must be trivial as in Case 1.

Hence $\inv_p \calA(-)$ is trivial also by \eqref{diag}.

\subsubsection{$p=2$}

Let $v=v_2$ be the unique place lying above 2. Let $P=[w:x:y:z]\in X_0(\Q_2)$. We can homogenize so that $w,x,y,z\in\Z_2$ and one of them is in $\Z_2^\times$. We can assume either $w$ or $x$ is a $2$-adic unit, since if $2\mid w,x$, then $2\mid y$ so the 2-adic valuation of $w^2+3x^6-97y^6$ is either 2 or $\geq6$. It clearly cannot be 2, so that means 2-adic valuation of $97\cdot28\cdot8z^6$ is 6 or greater which means $2|z$, a contradiction. Now a MAGMA computation shows that $f(P)$ is a $v$-adic unit and is always congruent to $1 \mod 8\calO_v$. Then $[f(P),28]_v$ is trivial since $f(P)$ is a cube in $K_v$.

\subsubsection{$p=7$}

Let $v=v_\mathfrak{p}$ be a place corresponding to a prime $\mathfrak{p}$ lying above $7$. Then $K_v\cong\Q_7$. Let $P=[w:x:y:z]\in X_0(\Q_7)$. We can homogenize so that $w,x,y,z\in\Z_7$. It is clear that we can assume either $w$ or $x$ is a 7-adic unit. There are two cases:

Case 1: $\mathfrak{p}$ divides neither $w-\sqrt{-3}x^3$ nor $w+\sqrt{-3}x^3$. A MAGMA computation shows that this is not possible.

Case 2: $\mathfrak{p}$ divides either $w-\sqrt{-3}x^3$. Note that 7 cannot divide $w+\sqrt{-3}x^3$ in this case. Then $v_\mathfrak{p}(w-\sqrt{-3}x^3)=v_\mathfrak{p}(w^2+3x^6)=v_\mathfrak{p}(97y^6+97\cdot28\cdot8z^6)$. Observe that $v_\mathfrak{p}(97y^6+97\cdot28\cdot8z^6)\equiv 0,1\mod 6$.

Subcase 1: Valuation is $0$ mod $6$. This means that $97y^6+97\cdot28\cdot8z^6\equiv97\cdot7^{6n}\mod 7^{6n+1}$. A MAGMA computation also shows that $w+\sqrt{-3}x^3\equiv 3,-3\mod 7$. Thus
\begin{alignat*}{2}
f(P)&=\frac{w-\sqrt{-3}x^3}{w+\sqrt{-3}x^3}\\
&=\frac{(w-\sqrt{-3}x^3)(w+\sqrt{-3}x^3)}{(w+\sqrt{-3}x^3)^2}\\
&=\frac{97y^6+97\cdot28\cdot8z^6}{(w+\sqrt{-3}x^3)^2}\\
&\equiv\frac{97\cdot7^{6n}}{2}&\pmod {7^{6n+1}}\\
&\equiv3\cdot7^{6n}&\pmod {7^{6n+1}}
\end{alignat*}
So we get
$$[f(P),28]_v=[3,28]_v=[3,7]_v+[3,4]_v=[3,7]_v=1/3$$
by (64) of \cite{ctksdiag}. (Note $[3,4]_v=0$ since $3$ becomes a cube in $K_v(\sqrt[3]{4})$)

Subcase 2: Valuation is $1$ mod $6$. This means that $97y^6+97\cdot28\cdot8z^6\equiv97\cdot28\cdot8\cdot7^{6n}\mod 7^{6n+2}$. We can carry out the same computation as above to get
$$[f(P),28]_v=[28\cdot3,28]_v=[3,28]_v=1/3.$$

Case 3: $\mathfrak{p}$ divides $w+\sqrt{-3}x^3$. Note that 7 cannot divide $w-\sqrt{-3}x^3$ in this case. We carry out the same procedure as case 2 except now
\begin{align*}
f(P)&=\frac{w-\sqrt{-3}x^3}{w+\sqrt{-3}x^3}\\
&=\frac{(w-\sqrt{-3}x^3)^2}{(w-\sqrt{-3}x^3)(w+\sqrt{-3}x^3)}\\
&=\frac{(w-\sqrt{-3}x^3)^2}{97y^6+97\cdot28\cdot8z^6}\\
\end{align*}
Here again $w-\sqrt{-3}x^3\equiv 3,-3\mod 7$, so we get the inverse of the solutions from case 2. Thus
$$[f(P),28]_7=2/3.$$

Combining the above, we get that $\inv_p\calA(-)$ is 0 for all $p\neq 7$ and $\inv_7 \calA(-)$ has value $1/3$ or $2/3$. Hence we have a Brauer--Manin obstruction coming from $\calA$.\end{proof}

\begin{remark}\label{infrem} The proof only uses the congruence class of 97 modulo several primes. Hence one can obtain infinitely many examples as follows. For any prime $p$, let
$$X_p:= w^2=-3x^6+py^6+p\cdot28\cdot8z^6.$$
It is clear that $X_p$ has real points. Assume that $p$ is congruent to 97 modulo some sufficiently high enough power of every prime $q\leq23$ (The existence of infinitely many such $p$ follows from Dirichlet's theorem on primes in arithmetic progressions). This ensures that $X_p$ has local points for all primes $q\leq23$. Then the Weil conjectures give that for any prime $q\geq 23,q\neq p$, there is a smooth $\F_q$-point on the curve $\{w=0\}$ which can be lifted to a $\Q_q$-point. Lastly, $\sqrt{-3}\in\Q_{97},\Q_p$ since $97,p\equiv 1\pmod 3$, so clearly there are $\Q_{97}$ and $\Q_p$ points. Hence $X_p$ is everywhere locally soluble. Furthermore assume $p\equiv1\pmod 8$ to ensure that the invariant maps over $\Q_2$ will carry out same as in the proof. Then $X_p$ will be a counterexample to the Hasse principle given by a 3-torsion Brauer class in the algebraic Brauer group.\end{remark}

\bibliographystyle{alpha}
\bibliography{pat}

\end{document}